\documentclass[11pt,reqno]{amsart}

\usepackage{amssymb, amsmath, amsthm,amscd,mathabx,amsxtra}
\usepackage[colorlinks,plainpages,backref=page,urlcolor=blue]{hyperref}
\usepackage[alphabetic,lite]{amsrefs}
\usepackage{mathtools}
\usepackage{verbatim}
\usepackage{amscd}
\usepackage{tikz}
\usetikzlibrary{cd,calc,arrows}
\usepackage{mathrsfs}
\usepackage{cases}
\usepackage{verbatim}
\usepackage[shortlabels]{enumitem}

\def\PP{{\textbf P}}
\def\OO{\mathcal{O}}

\def\cD{\mathcal{D}}
\def\cB{\mathcal{B}}
\def\cA{\mathcal{A}}
\def\cV{\mathcal{V}}

\def\P{\mathcal{P}}

\def\K{\mathcal{K}}
\def\L{\mathcal{L}}

\def\cM{\mathcal{M}}
\def\cR{\mathcal{R}}
\def\rr{\overline{\mathcal{R}}}

\def\cC{\mathcal{C}}

\def\mm{\overline{\mathcal{M}}}
\def\ss{\overline{\mathcal{S}}}

\def\dd{\overline{\mathcal{D}}}
\def\pm{\widetilde{\mathcal{M}}}
\def\pg{\widetilde{\mathcal{G}}}
\def\pr{\widetilde{\mathcal{R}}}
\def\prg{\widetilde{\mathcal{RG}}}

\def\ddd{\overline{\mathcal{D}'}}
\def\rk{\operatorname{rk}}
\newcommand{\dblq}{{/\!/}}

\newtheorem{theorem}{Theorem}[section]
\newtheorem*{theorem*}{Theorem}
\newtheorem*{problem*}{Problem}
\newtheorem{lemma}[theorem]{Lemma}

\newtheorem{corollary}[theorem]{Corollary}
\newtheorem*{corollary*}{Corollary}

\theoremstyle{definition}

\newtheorem*{definition*}{Definition}

\newtheorem*{remark*}{Remark}
\newtheorem*{notation*}{Notation}

\begin{document}

\title[The uniruledness of the Prym moduli space of genus $9$]
{The uniruledness of the Prym moduli space of genus $9$}

\author[G. Farkas]{Gavril Farkas}
\address{Farkas:  Humboldt-Universit\"at zu Berlin, Institut f\"ur Mathematik, \hfill \newline
\indent Unter den Linden 6,
10099 Berlin, Germany}
\email{{\tt farkas@math.hu-berlin.de}}

\author[A. Verra]{Alessandro Verra}
\address{Verra: Universit\'a Roma Tre, Dipartimento di Matematica, Largo San Leonardo Murialdo \hfill
 \newline \indent 1-00146 Roma, Italy}
 \email{{\tt verra@mat.uniroma3.it}}

\begin{abstract}
We show that the moduli space $\rr_9$ of Prym curves of genus $9$ is uniruled. This is the largest genus for which such a result is known to hold.
\end{abstract}

\maketitle

The moduli space $\cR_g$ parametrizing pairs $[C, \eta]$, where $C$ is a smooth curve of genus $g$ and $\eta\in \mbox{Pic}^0(C)[2]$ is a (non-trivial) $2$-torsion point in the Jacobian of $C$ has traditionally received considerable attention in the context of finding a uniformization of the moduli space $\cA_{g-1}$ of principally polarized abelian varieties of dimension $g-1$ via the Prym map $P_g\colon \cR_g\rightarrow \cA_{g-1}$, see \cites{B,CMGHL, Don}. In particular, in small genus when $P_g$ is a dominant map any result on the birational geometry of $\cR_g$ has direct consequences for $\cA_{g-1}$. It is known that $\cR_g$ is rational for $g=2,3,4$ (see \cite{Dol,Ca}), whereas $\cR_5$ is unirational \cites{ILS, V2}. In genus $6$  the Prym map $P_6\colon \cR_6\rightarrow \cA_5$ is finite of degree $27$ and there are at least two fundamentally different ways of showing that $\cR_6$ is unirational, see \cites{Don, V1}. Using Nikulin surfaces (that is, $K3$ surfaces endowed with a symplectic involution), we showed that $\cR_7$ is unirational as well \cites{FV1}, whereas $\cR_8$ is uniruled, see \cites{FV2}.

\vskip 4pt

The Prym moduli space $\cR_g$ admits a Deligne-Mumford compactification $\rr_g:=\mm_g\bigl(\mathcal{B} \mathbb Z_2\bigr)$, which can be interpreted either as the moduli space of stable maps from curves of genus $g$ to the classifying stack $\mathcal{B}\mathbb Z_2$, or in the spirit of Cornalba's work \cite{Cor}, as the stack of stable Prym curves of genus $g$, see \cites{BCF, FL}. We denote by $\pi\colon \rr_g\rightarrow \mm_g$ the morphism forgetting the Prym structure. It has been shown in \cite{FL} that $\rr_g$ is a variety of general type whenever $g\geq 14$ and $g\neq 15,16$. Bruns \cite{Br} extended this result and showed that $\rr_{15}$ is of general type. By making use of non-abelian Brill-Noether theory and tropical methods, it has been recently shown in \cite{FJP2} that $\rr_{13}$ is of general type as well.
By combining  results from \cites{CEFS, FL}, it also follows  that the Kodaira dimension of $\rr_{12}$ is non-negative. On the other hand the case of $\rr_{16}$ remains open, due to the highly surprising expected failure of the Prym-Green Conjecture in genus $16$, as discussed at length in \cite[Proposition 4.4]{CEFS}\footnote{It is precisely the failure locus of the Prym-Green Conjecture on syzygies of Prym-canonical curves which is used to show in \cite{FL} that $\rr_g$ is of general type for a given even genus $g$.}. However, we expect that $\rr_{16}$ is of general type as well. The aim of this paper is to establish the following result:
\begin{theorem}\label{thm:main}
The Prym moduli space $\rr_9$ is uniruled.
\end{theorem}
Note that $9$ is the largest genus $g$ for which $\rr_g$ is known to have negative Kodaira dimension. This leaves $\rr_{10}$ and $\rr_{11}$ as the only Prym moduli spaces where there is not even a conjectural description concerning their birational nature. In comparison, the Kodaira dimension of $\mm_g$ remains unknown for $g=17, \ldots, 21$, see \cites{HM, EH, FJP1}, whereas the Kodaira dimension of both components
$\overline{\mathcal{S}}_g^+$ and $\overline{\mathcal{S}}_g^-$ of the moduli space of spin curves is known for all $g$, see \cite{FV3}. Note that $\ss_9^+$ is known to be of general type, whereas $\ss_9^-$ is uniruled \cite{FV3} and $g=9$ is the first case where $\ss_g^-$ is not known to be unirational. It remains an open question, whether one can find a unirational parametrization for $\rr_9$.

\vskip 4pt

The proof of Theorem \ref{thm:main} relies in an essential way on the study of the following effective divisor of Brill-Noether type on the moduli space $\cR_9$
$$\cD_9:=\Bigl\{[C, \eta]\in \cR_9: \exists L\in W^2_8(C) \ \mbox{ such that } \ H^0(C,L\otimes \eta)\neq 0\Bigr\},$$
where $W^2_8(C)$ is the Brill-Noether variety of line bundles $L\in \mbox{Pic}^8(C)$ with $h^0(C,L)\geq 3$.
By standard Brill-Noether theory, a general curve $C$ of genus $9$ has a \emph{finite} number of linear systems $L\in W^2_8(C)$ and  the geometric condition defining $\cD_9$ can be reformulated as requiring that the line bundle $L\otimes \eta\in \mbox{Pic}^8(C)$ belongs to the theta divisor of $C$. As such, this condition is obviously divisorial in moduli. On the one hand, we can extend the determinantal structure defining $\cD_9$ to the boundary of $\rr_9$ and thus compute the class of the closure  of  $\cD_9$ in $\rr_9$ and we obtain the following formula, see also Theorem \ref{thm:divisor_class}:
\begin{theorem}\label{thm:div_intro}
The class of the closure of the Brill-Noether divisor in $\rr_9$ is given by
\begin{equation}\label{eq:classd9}
[\overline{\cD}_9]=366\lambda-52\bigl(\delta_0'+\delta_0''\bigr)-\frac{187}{2}\delta_0^{\mathrm{ram}}-E,
\end{equation}
where $E$ is an effective combination of boundary divisors of $\rr_9$ that \emph{does not} contain the boundary classes $\delta_0'$ and $\delta_0^{\mathrm{ram}}$.
\end{theorem}
Here we use the standard notation from \cite{FL} for the generators of $\mbox{Pic}(\rr_g)$ and refer to Section \ref{sec:classd} for details. In particular $\lambda$ is the Hodge class on $\rr_g$, whereas the other boundary classes appearing in the statement of Theorem \ref{thm:div_intro} are defined by the formula $\pi^*(\delta_0)=\delta_0'+\delta_0''+2\delta_0^{\mathrm{ram}}$, where $\delta_0\in \mbox{Pic}(\mm_g)$ is the class of the divisor of irreducible stable curves.

\vskip 4pt

The key point in our argument is that each component of the divisor $\dd_9$ is uniruled and, subject to a delicate transversality assumption which is established in \ref{subsec:trans}, we can exhibit a sweeping rational curve $R\subseteq \overline{\cD}_9$ satisfying both
\begin{equation}\label{eq:intnumbers}
R\cdot K_{\rr_9}<0 \ \  \mbox{ and } \ \ R\cdot \overline{\cD}_9>0.
\end{equation}
If one knew that $\dd_9$ was irreducible, the inequalities (\ref{eq:intnumbers}) would immediately imply that $K_{\rr_9}$ cannot be \emph{pseudo-effective}, that is, a limit of effective $\mathbb Q$-divisor classes and then using \cite{BDPP} (see also \cite[Proposition 5.1]{FV4}), it would follow that $\rr_9$ is uniruled. What in fact we \emph{can} show (Theorem \ref{thm:anycomponent}) is that there exists an irreducible component $\ddd$ of $\dd_9$ which is swept by rational curves $R\subseteq \ddd$ such that
$$R\cdot K_{\rr_9}<0 \  \mbox{  and } \ R\cdot \ddd \geq 0,$$ which, as explained, implies via \cite{BDPP} (or \cite{FV4} \emph{loc.cit.}) that $\rr_9$ is uniruled. We are able to carry this out, while stopping short of completely determining the class of $\ddd$.

\vskip 4pt

We now  discuss the construction of the pencil $R$ of Prym curves sweeping every component of the divisor $\overline{\cD}_9$. We start with a sufficiently general point $[C, \eta]\in \cD_9$. Therefore there exists a linear system $L\in W^2_8(C)$ such that $H^0(C, L\otimes \eta)\neq 0$, which by Riemann-Roch also implies that $H^0(C, \overline{L})\neq 0$, where $\overline{L}:=\omega_C\otimes L^{\vee}\in W^2_8(C)$ is the residual linear system. We write
\begin{equation}\label{eq:tgpoints}
\overline{L}\otimes \eta\cong \OO_C(y_1+ \cdots + y_8),
\end{equation}
where $y_1, \ldots, y_8\in C$. We may assume that the points $y_1, \ldots, y_8$ are mutually distinct and that  $L$ is globally generated and induces a birational map
$\varphi_L\colon C\rightarrow \Gamma\subseteq \PP^2$, where $\Gamma$ is a nodal plane octic (this last claim follows from e.g. \cite[Theorem 2]{EH2}).  We denote by $o_1, \ldots, o_{12}\in \PP^2$ the nodes of $\Gamma$ and set $\{x_i, x_i'\}=\varphi_L^{-1}(o_i)$ for $i=1, \ldots, 12$.  Because of the generality of $[C, \eta]$ we may assume that the sets
$\{x_1, x_1', \ldots, x_{12}, x_{12}'\}$ and $\{y_1, \ldots, y_8\}$ are disjoint.  By adjunction, we have
$$\overline{L}\cong \OO_C(4)\Bigl(-\sum_{i=1}^{12} (x_i+x_i')\Bigr),$$
where $L=\OO_C(1)$, that is, the linear system $|\overline{L}|$ is cut out on $\Gamma$ by a quartic plane curve passing through the nodes $o_1, \ldots, o_{12}$.
From (\ref{eq:tgpoints}), we obtain that
\begin{equation}\label{eq:tg2}
\OO_C(2y_1+\cdots+2y_8)\cong \overline{L}^{\otimes 2}\cong \OO_C(8)\Bigl(-2\sum_{i=1}^{12} (x_i+x_i')\Bigr),
\end{equation}
which amounts to saying that there exists an octic curve $\Gamma'\subseteq \PP^2$ nodal at the points $o_1, \ldots, o_{12}$ and tangent to $\Gamma$ at the points $y_1, \ldots, y_8$. Quite remarkably, the curve $\Gamma'$ has the same numerical characteristics as $\Gamma$ and we can consider the pencil $\{\Gamma_t\}_{t\in \PP^1}$ of octics spanned by $\Gamma$ and $\Gamma'$. Each curve in this pencil has singularities at the points $o_1, \ldots, o_{12}$ and passes through the points $y_1, \ldots, y_8$, therefore its normalization is a curve of genus $9$.  Because of condition  (\ref{eq:tg2}), we can lift this pencil to a pencil $R$ of Prym curves in $\rr_9$, by taking
\begin{equation}\label{eq:pencil1}
R:=\Bigl\{[C_t, \ \eta_t=\omega_{C_t}(-1)(-y_1-\cdots-y_8)\bigr]\Bigr\}_{t\in \PP^1}\subseteq \rr_9,
\end{equation}
where $\varphi_t\colon C_t\rightarrow \Gamma_t\subseteq \PP^2$ is the normalization map.

\vskip 4pt

We will establish in Section \ref{sec:uniruledd} that we can choose a Prym curve $[C, \eta]\in \cD_9$ filling up a codimension one subvariety of $\rr_9$, that is, an irreducible component $\ddd$ of $\overline{\cD}_9$, such that every curve $C_t$ in the pencil $R$ passing through the point $[C, \eta]$ is \emph{irreducible} and \emph{nodal}. We then compute the intersection of $R$ with the generators of $\mbox{Pic}(\rr_9)$. First observe that because $\Gamma$ and $\Gamma'$ share a common tangent line at each of the points $y_1,\ldots, y_8$, there will be precisely one curve in the pencil $R$
that has a nodal singularity at $y_i$. The corresponding Prym curve lies in the boundary divisor $\Delta_0^{\mathrm{ram}}$ (which can be viewed as the ramification divisor of the finite map $\pi\colon \rr_9\rightarrow \mm_9$). Moreover, these are the only points of intersection of $R$ and
$\Delta_0^{\mathrm{ram}}$ and at each of these points the intersection is transverse, which shows that $R\cdot \delta_0^{\mathrm{ram}}=8$. Calculation to be performed in Theorem \ref{thm:int_numbers3}  then imply that
\begin{equation}\label{eq:intnumbers2}
R\cdot \lambda=9, \ R\cdot \delta_0'=47, \ R\cdot \delta_0''=0 \ \mbox{ and } \ R\cdot \delta_i=R\cdot \delta_{9-i}=R\cdot \delta_{i:9-i}=0, \mbox{ for } i=1, \ldots, 4,
\end{equation}
which leads to the following intersection number with the canonical class
\begin{equation}\label{eq:canonical}
R\cdot K_{\rr_9}=R\cdot \Bigl(13\lambda-2(\delta_0'+\delta_0'')-3\delta_0^{\mathrm{ram}}-\cdots \Bigr)=13\cdot 9-2\cdot 47-3\cdot 8=-1.
\end{equation}
The rational curve $R$ is a moving curve for the divisor $\ddd$. Since in Section \ref{sec:uniruledd} (Theorem \ref{thm:anycomponent}) we also show that any curve $R\subseteq \rr_9$ having the intersection numbers given by (\ref{eq:intnumbers2}) has to intersect non-negatively any divisor disjoint from a generic pencil of Prym curves on a general Nikulin surface, we conclude that $R\cdot \ddd \geq 0$. This implies that $K_{\rr_9}$ is not pseudo-effective, thus proving Theorem \ref{thm:main}.

\vskip 4pt

{\small{{\bf Acknowledgments:} We have greatly benefitted from discussions with Margherita Lelli-Chiesa related to this circle of ideas.

Farkas supported by the DFG Grant \emph{Syzygien und Moduli} and by the ERC Advanced Grant SYZYGY. This project has received funding from the European Research Council (ERC) under the EU Horizon 2020  program (grant agreement No. 834172).}} This work was started when Verra was an Invited Professor of the Berlin Mathematical School.

\section{The Brill-Noether divisor on $\rr_9$}\label{sec:classd}
In this Section, after recalling basic facts about the moduli space $\rr_g$, we  compute the class of the divisor $\overline{\cD}_9$, then establish various facts about the class of every irreducible component of $\dd_9$, which will prove to be essential in the proof of Theorem \ref{thm:main}. It turns out that the formula for the class $[\dd_9]$ is stated without proof as part of \cite[Theorem 0.4]{FL}. Due to the importance  this class plays in the course of our argument, and because the formula in \emph{loc.cit.} contains several missprints, we shall present all details of the calculation of $[\dd_9]$.

\vskip 4pt

Recall \cites{B, BCF, FL} that $\rr_g$ denotes the moduli stack of stable Prym curves of genus $g$, that is, consisting of triples $[X, \eta, \beta]$, where $X$ is a quasi-stable genus $g$ curve, $\eta$ is a locally free sheaf of total degree $0$ on $X$ such that $\eta_{|E}\cong \OO_E(1)$ for every smooth rational component $E\subseteq X$ with $|E\cap \overline{X\setminus E}|=2$ (such a component being called \emph{exceptional}) and $\beta\colon \eta^{\otimes 2}\rightarrow \OO_X$ is a sheaf morphism that is non-zero along each non-exceptional component of $X$. There exists a finite branch map $\pi\colon \rr_g\rightarrow \mm_g$ assigning to a triple $[X, \eta, \beta]$ as above the stabilization of $X$, obtained from $X$ by contracting all of its exceptional components.

\vskip 4pt

The Picard group of $\rr_g$ is freely generated by the Hodge class $\lambda$ and the boundary classes $\delta_0', \delta_0''$, $\delta_0^{\mathrm{ram}}$ and $\delta_i, \delta_{g-i}, \delta_{i:g-i}$, where $i=1, \ldots, \lfloor \frac{g}{2}\rfloor$. Denoting by $\delta_0\in \mbox{Pic}(\mm_g)$ the class of the locus of irreducible stable curves and by $\delta_i\in \mbox{Pic}(\mm_g)$ for $i=1, \ldots, \bigl \lfloor \frac{g}{2}\bigr\rfloor$ the class of the closure of the locus of the union of two smooth curves of genus $i$ and $g-i$ meeting at one point, one has the following relations:
\begin{equation}
\pi^*(\delta_0)=\delta_0^{'}+\delta_0^{''}+2\delta_{0}^{\mathrm{ram}} \   \    \mbox{  and } \  \   \pi^*(\delta_i)=\delta_i+\delta_{g-i}+\delta_{i:g-i}, \mbox{ for } i\geq 1.
\end{equation}
We now recall the meaning of the  classes  $\delta_0^{'}:=[\Delta_0^{'}], \ \delta_0^{''}:=[\Delta_0^{''}]$ respectively
 $\delta_0^{\mathrm{ram}}:=[\Delta_0^{\mathrm{ram}}]$, while referring to \cite{FL} for details.  If we  fix a general point $[C_{xy}]\in \Delta_0$ induced by a $2$-pointed curve $[C, x, y]\in \cM_{g-1, 2}$ and the normalization map $\nu\colon C\rightarrow C_{xy}$, where $\nu(x)=\nu(y)$, a general point of the irreducible divisor $\Delta_0'$ (respectively of $\Delta_0''$) corresponds to a stable Prym curve $[C_{xy}, \eta]$, where $\eta\in \mathrm{Pic}^0(C_{xy})[2]$ and $\nu^*(\eta)\in \mathrm{Pic}^0(C)$ is non-trivial
(respectively trivial). A general point of $\Delta_{0}^{\mathrm{ram}}$ is of the form $[X, \eta]$, where $X:=C\cup_{\{x, y\}} \PP^1$ is a quasi-stable curve  and  $\eta\in \mathrm{Pic}^0(X)$ satisfies $\eta_{\PP^1}\cong \OO_{\PP^1}(1)$ and $\eta_C^{\otimes 2}\cong \OO_C(-x-y)$. Note that $\Delta_0^{\mathrm{ram}}$ corresponds generically to $1$-nodal irreducible curves where the Prym structure is not free at the node of the underlying curve.

Applying the Hurwitz formula to the branch cover $\pi\colon \rr_g\rightarrow \mm_g$, we also have
\begin{equation}\label{eq:canonical1}
K_{\rr_g}=13\lambda-2\bigl(\delta_0'+\delta_0''\bigr)-3\delta_0^{\mathrm{ram}}-3\bigl(\delta_1+\delta_{g-1}+\delta_{1:g-1}\bigr)-2\sum_{i=2}^{\lfloor \frac{g}{2}\rfloor}\bigl(\delta_i+\delta_{g-i}+\delta_{i:g-i}\bigr).
\end{equation}

\subsection{The divisor $\cD_9$.} We now specialize to the case of $\rr_9$. From the Brill-Noether Theorem \cite{ACGH} it follows that for a general curve $[C]\in \cM_9$ the Brill-Noether variety $G^2_8(C)$ is finite and consists of $42$ distinct points. We denote by $\mathcal{G}^2_8\rightarrow
\cM_9$ the Deligne-Mumford stack of linear systems classifying pairs $[C, \ell]$, where $[C]\in \cM_9$
and $\ell=(L, V)\in G^2_8(C)$.

\vskip 4pt

We introduce the \emph{Brill-Noether divisor} on $\cR_9$
$$\mathcal{D}_{9}:=\Bigl\{[C, \eta]\in \cR_9: \exists L\in W^2_{8}(C)\
\mbox{ such that } \ h^0(C, L\otimes \eta)\geq 1\Bigr\}.$$
It follows from \cite[Theorem 2.3]{FL}  that $\mathcal{D}_{9}$ is an effective divisor on $\cR_9$.

In order to realize $\cD_9$ as the degeneracy locus of two vector bundles of the same rank over $\cR_9$, for a pair $[C,L]\in \mathcal{G}^2_8$ such that $L$ is globally generated and $h^0(C,L)=3$, let $M_L$ be the rank $2$ syzygy bundle defined by the exact sequence
\begin{equation}
0\longrightarrow M_L \longrightarrow H^0(C,L)\otimes \OO_C \longrightarrow L \longrightarrow 0.
\end{equation}
Tensoring the resulting exact sequence with $\eta$, for a triple $[C, L, \eta]\in  \mathcal{RG}^2_8:=\mathcal{G}^2_8\times_{\cM_9} \cR_9$  the condition $H^0(C, L\otimes \eta)\neq 0$ is then equivalent to the non-injectivity of the following map

\begin{equation}\label{eq:detstr}
\chi_{C,L,\eta} \colon  H^1\bigl(C, M_L\otimes \eta\bigr)\longrightarrow H^0(C, L\otimes \eta)\otimes H^1(C, \eta).
\end{equation}

Since  $H^0\bigl(C, M_L\otimes \eta)=0$, by Riemann-Roch $h^1\bigl(C, M_L\otimes \eta)=-\mbox{deg}(M_L)+2\cdot(g-1)=\mbox{deg}(L)+2\cdot 8=24$, whereas clearly $\mbox{dim } H^0(C,L)\otimes H^1(C,\eta)=h^0(C,L)\cdot h^0(C, \omega_C\otimes \eta)=3\cdot 8=24$, that is, $\chi_{C,L,\eta}$ is a map between vector spaces of the same dimension. We globalize the description (\ref{eq:detstr}) to a morphism of vector bundles of the same rank over the moduli stack of Prym curves.

\vskip 4pt

Let $\cM_9^{\circ}\subseteq \cM_9$ be the open substack
of smooth curves $C$ of genus $9$ such that $G_{7}^2(C)=
\emptyset$.  It is easy to see that
$\mbox{codim}(\cM_9-\cM_g^{\circ}, \cM_9)\geq 2$. In particular, $h^0(C, L)=3$, for every $[C]\in \cM_9^{\circ}$ and $L\in W^2_8(C)$. We denote by
$\Delta_0^{\circ}\subseteq \Delta_0\subset \mm_9$ the locus of nodal curves
$[C_{xy}=C/x\sim y]$, where $C$ is a smooth curve of genus $8$ such that $G^2_7(C)=\emptyset$ and $x, y\in C$ (note that we allow for the possibility $x=y$, in which case we attach an elliptic tail to $C$ at the point $x$).  Set
$$\pm_9:=\cM_9^{\circ}\cup \Delta_0^{\circ}\subseteq \mm_9 \ \  \mbox{ and } \ \  \pr_9:=\pi^{-1}(\widetilde{\cM}_9)\subseteq \rr_9.$$
Observe that $\mbox{Pic}(\pr_9)_{\mathbb Q}$ is freely generated by the classes $\lambda, \delta_0', \delta_0''$ and $\delta_0^{\mathrm{ram}}$.
Let
$$\pg^2_8\rightarrow \widetilde{\cM}_9$$
be the stack classifying pairs $[C, L]$, where $[C]\in \widetilde{\cM}_9$  and $L$ is a torsion free sheaf on $C$ having degree $8$ and with $h^0(C,L)=3$. Note that in the case of an element $[C_{xy}, L]$, where $C_{xy}$ is a $1$-nodal curve, the assumption $G^2_7(C)=\emptyset$ guarantees that $L$ is necessarily locally free, else $L=\nu_*(A)$, where $A\in G^2_7(C)$, where we recall that $\nu\colon C\rightarrow C_{xy}$ denotes the normalization map. For further details, we refer to \cite[page 770]{FL}.

We finally introduce the stack of linear series over Prym curves
$$\sigma\colon \mathfrak \prg ^2_8:=\pr_9 \times_{\widetilde{\cM}_9} \pg^2_8\rightarrow \widetilde{\cR}_9$$
and  consider the universal Prym curve of genus $9$ over it
$$f\colon \mathcal{C} \rightarrow \prg ^2_8.$$

Note that, if $t=\bigl[C\cup_{\{x,y\}}\PP^1, \eta, L]\in \sigma^{-1}(\Delta_0^{\mathrm{ram}})$, where $C$ is a smooth curve of genus $8$, then $f^{-1}(t)=C\cup_{\{x,y\}} \PP^1$, cf. \cite[1.1]{FL}.

At the level of $\mathcal{C}$ we have a universal \emph{Prym} line bundle
$\mathcal{P}$  and a \emph{Poincar\'e} line bundle which are characterized by the
property $\mathcal{P}_{|f^{-1}[X, \eta, \beta, L]}=\eta\in \mbox{Pic}^0(X)$ and $\mathcal{L}_{|f^{-1}[X, \eta, \beta, L]}=L\in \mbox{Pic}^8(X),$
for each point $[X, \eta, \beta, L]\in \prg^2_8$.

Following \cites{FL, F1}, we introduce the
codimension $1$ following tautological classes in $\mbox{Pic}\bigl(\prg_8^2\bigr)$:

\begin{equation}\label{tautological}
\mathfrak{a}:=f_*\bigl(c_1(\L)^2\bigr) \ \  \mbox{ and } \  \ \mathfrak{b}:=f_*\bigl(c_1(\L)\cdot
c_1(\omega_{f})\bigr).
\end{equation}

We shall also need the tautological rank $3$ vector bundle on $\prg_8^2$
$$\mathcal{V}:=f_*(\L).$$
The fact that $\mathcal{V}$ is locally free follows from Grauert's theorem, since as we already explained, $h^0(C,L)=3$, for every $[C,L]\in \pg^2_8$.

We record the following formulas describing the push-forward of these tautological classes under the generically finite morphism $\sigma$, see \cite{F1}:
\begin{equation}\label{eq:pushforward}
\sigma_*(\mathfrak{a})=-564\lambda+83\bigl(\delta_0'+\delta_0''+2\delta_{\mathrm{ram}}\bigr), \ \mbox{ and } \ \  \sigma_*(\mathfrak{b})=252\lambda-21\bigl(\delta_0'+\delta_0''+2\delta_0^{\mathrm{ram}}\bigr).
\end{equation}
The class $\sigma_*\bigl(c_1(\cV)\bigr)$ can also be determined, see \cite{FL}, but it will not be used in what follows.

\vskip 4pt

We consider the global syzygy bundle on $\cC$ defined by the exact sequence
$$
0\longrightarrow \cM\longrightarrow f^*(\cV)\longrightarrow \L\longrightarrow 0.
$$

From this sequence,  we obtain the following formulas which we record:
\begin{equation}\label{eq:chernclassesM}
c_1(\cM)=f^*c_1(\cV)-c_1(\L) \ \ \mbox{ and } \ \ c_2(\cM)=f^*c_2(\cM)+c_1^2(\L)-f^*c_1(\cV)\cdot c_1(\L).
\end{equation}

We then introduce the following vector bundles over $\prg_8^2$:
$$\cA:=R^1f_*\bigl(\cM\otimes \P\bigr) \ \ \mbox{ and }  \ \ \cB:=f^*\bigl(\cV \otimes R^1f_*(\P)\bigr),$$
having fibres $\cA\bigl([X,\eta,\beta,L]\bigr)=H^1(X,M_L\otimes \eta)$ and $\cB\bigl([X, \eta, \beta, L]\bigr)=H^0(X,L)\otimes H^1(X, \eta)$.
Note that both  $\cA$ and $\cB$ are locally free sheaves of rank $24$ and there exists a natural morphism

\begin{equation}\label{eq:sheaf_morphism}
\chi\colon \cA\rightarrow \cB,
\end{equation}
which over a point $[C,\eta,L]$ corresponding to a smooth curve $C$ globalizes the maps (\ref{eq:detstr}). Let $\widetilde{Z}$ be the degeneracy locus of the morphism $\chi$ and $Z:=\tilde{Z}\cap (\pi\circ \sigma)^{-1}\bigl(\cM_9^{\circ})$. Therefore, $\sigma(Z)$  coincides over $\cR_9^{\circ}=\pi^{-1}(\cM_9^{\circ})$ with the divisor $\cD_9$.

\begin{theorem}\label{thm:classdege}
The class of the degeneracy locus $\widetilde{Z}$ of the morphism $\chi\colon \cA\rightarrow \cB$ equals
$$[\widetilde{Z}]=c_1(\cB-\cA)=-\lambda-\frac{\mathfrak{a}}{2}+\frac{\mathfrak{b}}{2}+\frac{1}{4}\sigma^*(\delta_0^{\mathrm{ram}}).$$
\end{theorem}
\begin{proof} Calculating the class of $\cB$ is straightforward. Using \cite[Proposition 1.7]{FL} we have that $c_1\bigl(R^1f_*(\P)\bigr)=-c_1\bigl(f_*(\omega_f\otimes \P^{\vee})\bigr)=-\lambda+\frac{1}{4} \sigma^*(\delta_0^{\mathrm{ram}})$, therefore
\begin{equation}\label{eq:classB}
c_1(\cB)=8f^*c_1(\cV)+3c_1\bigl(R^1f_*(\P)\bigr)=8f^*\bigl(c_1(\cV)\bigr)-3\lambda+\frac{3}{4}\sigma^*(\delta_0^{\mathrm{ram}}).
\end{equation}
In order to calculate $c_1(\cA)$ we apply Grothendieck-Riemann-Roch to the universal curve $f\colon \cC\rightarrow \prg^2_8$ and to the vector bundle $\cM\otimes \P$ and we write:
\begin{align*}
-c_1(\cA)=-c_1 \bigl(R^1f_*(\cM\otimes \P)\bigr)= f_*\Bigl[\Bigl(2+c_1(\cM\otimes \P)+\frac{c_1^2(\cM\otimes \P)-2c_2(\cM\otimes \P)}{2}\Bigr) \cdot \\ \Bigl(1-\frac{c_1(\omega_{f})}{2}+\frac{c_1^2(\omega_{f})
+[\mathrm{Sing}(f)]}{12}\Bigr)\Bigr]_2,
\end{align*}
where
$\mathrm{Sing}(f)\subseteq \cC$ denotes the codimension
$2$ singular locus (that is, the locus of nodes) of the universal curve $f$. Clearly
$f_*[\mathrm{Sing}(f)]=\sigma^*\bigl(\Delta_0'+\Delta_0''+2\Delta_0^{\mathrm{ram}}\bigr)$. To estimate the degree $2$ terms appearing in the right hand side of this formula, we use Mumford's formula \cite{HM} $f_*\bigl(c_1^2(\omega_f)\bigr)=12\lambda-\sigma^*(\delta_0'+\delta_0''+2\delta_0^{\mathrm{ram}})$, coupled with the formulas (\ref{eq:chernclassesM}), as well as with the following formulas, see \cite[Proposition 1.6]{FL}:

\begin{align*}
f_*\bigl(c_1^2(\P)\bigr)=-\frac{\delta_0^{\mathrm{ram}}}{2}, \ \ \ f_*(c_1(\omega_f)\cdot c_1(\P))=0, \ \ \ f_*\bigl(c_1(\omega_f)\cdot f^*c_1(\cV)\bigr)=16c_1(\cV),\\
f_*\bigl(c_1(\L)\cdot c_1(\P))=f_*\bigl(f^*(c_1(\cV)\cdot c_1(\P))=0, \  \ \ f_*\bigl(c_1(\L)\cdot f^*c_1(\cV)\bigr)=8c_1(\cV)
.
\end{align*}

Making also use of the formula $c_2(\cM\otimes \P)=c_2(\cM)+c_1^2(\P)+c_1(\cM)\cdot c_1(\P)$, we conclude that
\begin{equation}\label{eq:classA}
-c_1(\cA)=2\lambda+\frac{\mathfrak{b}}{2}-\frac{\mathfrak{a}}{2}-8c_1(\cV)-\frac{1}{2}\sigma^*(\delta_0^{\mathrm{ram}}).
\end{equation}
Combining (\ref{eq:classA}) and (\ref{eq:classB}), we find $[\widetilde{Z}]=c_1(\cB)-c_1(\cA)=-\lambda+\frac{\mathfrak{b}}{2}-\frac{\mathfrak{a}}{2}+\frac{1}{4}\sigma^*(\delta_0^{\mathrm{ram}})$, which finishes the proof.
\end{proof}

\section{Nikulin surfaces and the divisor $\dd_9$}\label{sec:nikulin}

In this section we show that a general pencil of Prym curves lying on a Nikulin surface is disjoint from the divisor $\dd_9$. We begin by recalling basic facts about the connection between Nikulin surfaces and Prym curves, our main references being \cites{FV1,vGS}.

\vskip 3pt

A \emph{Nikulin  surface} is a $K3$ surface $S$ endowed with a symplectic automorphism, or equivalently with a non-trivial double cover
$$f\colon \tilde S \rightarrow S $$
having a branch divisor
$N:=N_1 + \cdots + N_8$
consisting of $8$ disjoint smooth rational curves $N_i\subseteq S$.  The class $[N]$ is divisible by $2$ and we set $\mathfrak{e}:=\frac{1}{2}\bigl([N_1]+\cdots+[N_8]\bigr)\in \mathrm{Pic}(S)$   and define the \emph{Nikulin lattice} to be the rank $8$  lattice $\mathfrak{N}\subseteq \mbox{Pic}(S)$  generated by  $[N_1], \ldots, [N_8]$ and by $\mathfrak{e}$. A \emph{polarized}  Nikulin surface of genus $g$ is a Nikulin surface $f\colon \tilde{S} \rightarrow  S$ as above, together with a smooth curve $C\subset S$ of genus $g$ such that $C\cdot N_i=0$, for $i=1, \ldots,8$. There is an irreducible $11$-dimensional moduli space $\mathcal{F}_g^{\mathfrak{N}}$ of polarized Nikulin surfaces of genus $g$ and for a general such surface one has $\mbox{Pic}(S)\cong \mathbb Z\cdot [C]\oplus \mathfrak{N}$.
If $\tilde C := f^{-1}(C)$, then
$f_C:=f_{|\tilde C}\colon \tilde C \rightarrow C
$ is an \'etale double covering and
$\mathfrak{e}_C:=\mathcal O_C(\mathfrak{e})\in \mathrm{Pic}^0(C)$ is the non trivial $2$-torsion element defining the covering $f_C$.

\vskip 4pt

We define the \emph{Nikulin pencil} in $\rr_g$ to be the pencil $\Xi_g\subseteq \rr_g$ consisting of Prym curves $\bigl\{[C_t, \mathfrak{e}_{C_t}]\bigr\}_{t\in \PP^1}$  induced by a Lefschetz pencil $\bigl\{C_t\bigr\}_{t\in \PP^1}$ on a general polarized Nikulin surface $S$. The following formulas hold, see \cite[Proposition 1.4]{FV1}.
\begin{equation}\label{eq:Nikulin_pencil}
\Xi_g\cdot \lambda=g+1, \ \ \Xi_g\cdot \delta_0'=6g+2, \ \  \Xi_g\cdot \delta_0''=0\  \  \mbox{ and }\ \  \Xi_g\cdot \delta_0^{\mathrm{ram}}=8.
\end{equation}

All elements of $\Xi_g$ are irreducible curves, therefore the intersection of $\Xi_g$ with the other boundary divisors in $\rr_g$ are equal to zero.

\vskip 3pt

\subsection{Moduli of vector bundles on Nikulin surfaces}
Given a smooth $K3$ surface $S$,  the \emph{Mukai pairing} \cite{HL, Mu} on $H^{\bullet}(S)$ is defined by
$$(v_0, v_1, v_2)\cdot (w_0. w_1, w_2):=v_1\cdot w_1-v_2\cdot w_0-v_0\cdot w_2\in H^4(S, \mathbb Z)\cong \mathbb Z.$$
For a sheaf $F$ on $S$, let $v(F):=\Bigl(\rk(F), c_1(F), \chi(F)-\rk(F)\Bigr)$ be its Mukai vector.
For a polarization $H\in \mbox{Pic}(S)$, let $M_H(v)$ be the moduli space of $S$-equivalence classes of (Gieseker) $H$-semistable sheaves $F$ on $S$ with Mukai vector $v(F)=v$. Let $M_H^s(v)$ the open subset of $M_H(v)$ corresponding to $H$-stable sheaves. Then it is known that $M_H^s(v)$ is pure dimensional and $\mbox{dim } M_H^s(v)=v^2+2$. In particular, if $v(F)^2<-2$, then $F$ is not stable.

\vskip 3pt

Specializing now to the case when $S$ is a Nikulin surface, we now show that  the pencil $\Xi_9$ is disjoint from the divisor $\dd_9$. The following result uses essential input from M. Lelli-Chiesa:

\begin{theorem}\label{thm:nikulin}
Let $f\colon \tilde{S}\rightarrow S$ be a polarized Nikulin surface of genus $9$ with $\mathrm{Pic}(S)\cong \mathbb Z\cdot C \oplus \mathfrak{N}$ and let $\Xi_9\subseteq \rr_9$ be an induced Lefschetz pencil of Prym curves on $S$. Then $\Xi_9\cap \dd_9=\emptyset$.
\end{theorem}

\begin{proof}
Let $C\subseteq S$ be a curve in the polarization class $|C|$ of $S$, thus $C\cdot N_i=0$ for $i=1, \ldots, 8$. We fix $L\in W^2_8(C)$ such that $H^0(C, L\otimes \mathfrak{e}_C)\neq 0$ and set again $\overline{L}:=\omega_C\otimes L$.
Each curve in the linear system $|C|$ is Brill-Noether general, cf. \cite[Lemma 5.1]{FK}, in particular $L$ is globally generated and we consider the associated \emph{Lazarsfeld-Mukai bundle} \cite{La}
$$0\longrightarrow E_{C,L}^{\vee}\longrightarrow H^0(C,L)\otimes \OO_S\longrightarrow C\longrightarrow 0.$$
By dualizing one has the following exact sequence
\begin{equation}\label{eq:LM}
0\longrightarrow H^0(C,L)^{\vee}\otimes \OO_C\longrightarrow E_{C,L}\longrightarrow \overline{L}\longrightarrow 0.
\end{equation}
Observe that $E_{C,L}$ is globally generated. Tensoring the exact sequence (\ref{eq:LM}) by $\mathfrak{e}^{\vee}$, taking global section and using that $H^i\bigl(S, \mathfrak{e}^{\vee}\bigr)=0$ for all $i$ (cf. \cite[Lemma 1.3]{FV1}), we conclude that
$H^1\bigl(S, E_{C,L}\otimes \mathfrak{e}^{\vee}\bigr)\stackrel{\cong}\longrightarrow H^1\bigl(C,\overline{L}\otimes \mathfrak{e}_C^{\vee}\bigr)$, therefore $H^1\bigl(S, E_{C,L}\otimes \mathfrak{e}^{\vee}\bigr)\neq  0$. It follows that $\mbox{Ext}^1(E_{C,L}, \mathfrak{e})\neq 0$, that is, there exist a non-trivial extension
\begin{equation}\label{eq:extension}
0\longrightarrow \mathfrak{e} \longrightarrow E\longrightarrow E_{C,L}\longrightarrow 0.
\end{equation}
Since $c_1(E_{C,L})=[C]$, whereas $h^0\bigl(S, E_{C,L}\bigr)=h^0(C,L)+h^0(C, \overline{L})=6$ and $h^i\bigl(S, E_{C,L}\bigr)=0$ for $i=1,2$, we compute the Mukai vector
$$v(E)=\bigl(4, [C]+\mathfrak{e}, 2\bigr),$$
therefore $v(E)^2=-4<-2$, in particular the vector bundle $E$ is not stable and hence also not $\mu=\mu_C$-stable, that is, slope stable with respect to the polarization defined as $\mu_C(F):=\frac{c_1(F)\cdot C}{\rk(F)}$, for any coherent sheaf $F$ on $S$\footnote{Usually the Mumford-Takemoto $\mu:=\mu_C$-stability is defined with respect to an ample line bundle on $S$, but as pointed out in both \cite{GKP} and \cite[Remark 4.C.4]{HL} this assumption is too strong and can be replaced with the one that $C$ be big and nef, which is precisely the case at hand when $C$ is the polarization of a Nikulin surface.}

\vskip 3pt
Let $E_1$ be a maximally destabilizing subsheaf of $E$ of maximal rank $r:=\rk(E_1)\leq 3$. We may assume that the quotient $G=E/E_1$ is torsion free, hence $E_1$ is locally free and we have the following diagram:

\begin{equation}
\begin{tikzcd}[column sep=30pt,row sep=28pt]
& & E_1 \ar[d, hook] \ar[dr, "\phi" ] & & \\
0 \ar[r] & \mathfrak{e} \ar[r] & E \ar[r] \ar[d, two heads] & E_{C,L} \ar[r] & 0 \\
& & G & &
\end{tikzcd}
\end{equation}
We write $c_1(E_1)=a[C]+N'$, where $N'\in \mathfrak{N}$, in particular $C\cdot N'=0$. Then we write $\mu(E_1)=\frac{16a}{r}\geq \mu(E)=4$, which yields
$a\geq 1$.

We claim that $\mbox{Hom}(E_1,\mathfrak{e})=0$, therefore the image $\phi\in \mbox{Hom}(E_1, E_{C,L})$ of the injection $E_1\hookrightarrow E$ is a non-zero morphism. Indeed, assuming $0\neq h\in \mbox{Hom}(E_1, \mathfrak{e})$, set $E_1':=\mbox{Ker}(h)$ and write down an exact sequence
$$0\longrightarrow E_1'\longrightarrow E_1\stackrel{h}\longrightarrow \mathfrak{e}(-D)\otimes \mathcal{I}_{\xi/S}\longrightarrow 0,$$
where $\xi$ is a $0$-dimensional subscheme and $D$ is an effective divisor on $S$ respectively. In particular, $c_1(E_1')\cdot C=c_1(E_1)\cdot C+D\cdot C\geq c_1(E_1)\cdot C$ and therefore $\mu(E_1')=\frac{c_1(E_1')\cdot C}{\rk{E_1'}}>\frac{c_1(E_1)\cdot C}{\rk(E_1)}$, thus contradicting the maximality of $E_1$ among all destabilizing subsheaves of $E$.

\vskip 4pt

Therefore $\phi\colon E_1\rightarrow E_{C,L}$ is a non-zero morphism. Since the Lazarsfeld-Mukai bundle $E_{C,L}$ is easily shown to be $\mu$-semistable (the same proof as in the case of $K3$ surfaces of Picard number one treated in \cite{La} applies here as well), it follows that $c_1\bigl(E_{C,L}\otimes E_1^{\vee}\bigr)\cdot C\geq 0$, which yields $r=3$ and $a=1$, that is, $c_1(E_1')= C+N'$.

\vskip 3pt

Assume first $\phi$ is injective. Accordingly we have an exact sequence
$$0\longrightarrow E_1\longrightarrow E_{C,L}\longrightarrow Q\longrightarrow 0,$$
where $Q=\mathcal{I}_{\xi/S}(-N')$. Since $E_{C,L}$ is globally generated, $Q$ must be globally generated as well, which is impossible, for no sheaf of type $\mathcal{I}_{\xi}\bigl(b_1N_1+\cdots+b_8N_8\bigr)$ can be globally generated. In the general case, set $\K:=\mbox{Ker}(\phi)$, in particular $\K\hookrightarrow \cM$. Setting $Q:=\mbox{Coker}(\phi)$, we again conclude that $c_1(Q)\in \mathfrak{N}$, whereas $Q$ must be globally generated, which is a contradiction.

\vskip 3pt

We finish the argument by noticing that the same reasoning works for a $1$-nodal curve $C\in |C|$ and for a globally generated line bundle $L\in W^2_8(C)$, therefore $\Xi_9\cap \dd_9=\emptyset$.
\end{proof}

We are now in a position to finish the calculation of the class of the closure of the divisor $\cD_9$.

\begin{theorem}\label{thm:divisor_class}
One has the following formula for the closure $\widetilde{\cD}_9$ in $\pr_9$ of the divisor $\cD_9$:
$$[\widetilde{\cD}_9]=366\lambda-52\bigl(\delta_0'+\delta_0'')-\frac{187}{2}\delta_0^{\mathrm{ram}}-\alpha\cdot \delta_0''\in \mathrm{Pic}(\pr_9),$$
where $\alpha\geq 0$.
\end{theorem}

\begin{proof} We determined in Theorem \ref{thm:classdege} the class of the degeneracy locus $Z$ of the morphism $\chi\colon \cA\rightarrow \cB$ of vector bundles over $\prg^2_8$ defined in (\ref{eq:sheaf_morphism}). Using the formulas (\ref{eq:pushforward}), we conclude that $\sigma_*\bigl([\widetilde{Z}]\bigr)=366\lambda-52\bigl(\delta_0'+\delta_0'')-\frac{187}{2}\delta_0^{\mathrm{ram}}$. Furthermore, Theorem \ref{thm:nikulin} shows that the morphism $\chi$ is generically non-degenerate along (each component of) the boundary divisors $\sigma^*(\delta_0')$ and $\sigma^*(\delta_0^{\mathrm{ram}})$, that is, the class of $\sigma_*(\widetilde{Z})$ and that of $\widetilde{\cD}_9$ coincide up to a multiple of $\delta_0''$. The precise value of this multiple plays no further role in our argument.
\end{proof}

\section{The uniruled parametrization of $\overline{\cD}_9$}\label{sec:uniruledd}

We  now spell out the uniruled parametrization of the divisor $\dd_9$ that was sketched in the Introduction. Note that we do not establish that $\dd_9$ is irreducible, instead we  show that each irreducible component of $\dd_9$ is uniruled.

Let $[C,L]\in \mathcal{G}^2_8$ be a general element corresponding to a smooth curve $C$ of genus $9$ and a complete, globally generated linear system $L\in G^2_8(C)$. Let $\varphi_L\colon C\rightarrow \PP^2$ be the map induced by $|L|$ and set $\Gamma:=\varphi_L(C)$ and $\overline{L}:=\omega_C\otimes L^{\vee}$.  Since $C$ may be assumed to be Petri general, $H^1(C, \overline{L}^2)=0$, therefore $h^0\bigl(C, \overline{L}^2\bigr)=8$. Due to our generality assumptions, $\Gamma$ is a nodal octic and let $o_1, \ldots, o_{12}$ be its nodes and denote $\xi:=o_1+\cdots+o_{12}\in \mbox{Hilb}^{12}(\PP^2)$.  Set
$$\epsilon \colon X:=\mathrm{Bl}_{\xi}(\PP^2)\longrightarrow \PP^2$$
and denote by $E_1, \ldots, E_{12}$ the exceptional divisors corresponding to the $12$ points and by $h:=\epsilon^*(\OO_{\PP^2}(1))\in \mbox{Pic}(X)$ the hyperplane class.
Note that $X$ is also endowed with a regular degree $4$ cover
\begin{equation}\label{eq:degree4}
q:=\varphi_{|4h-E_1-\cdots-E_{12}|}\colon X\longrightarrow \PP^2.
\end{equation}

Conversely, we introduce the universal parameter space of $12$-nodal plane octics, that is,
$$\P:=\Bigl\{\bigl(o_1+\cdots+o_{12}, C\bigr):o_1+\cdots+o_{12}\in \mathrm{Hilb}^{12}(\PP^2), \ C\in \bigl|8h-2E_1-\cdots-2E_{12}\bigr|\Bigl\},$$
then observe that $\P\rightarrow \mbox{Hilb}^{12}(\PP^2)$ is generically a $\PP^8$-bundle and the map
$$ \begin{tikzcd}[column sep=22pt]
\P \dblq SL(3)\dashrightarrow \mathcal{G}^2_8, \ \ \bigl(o_1+\cdots+o_{12}, \
C\bigr)\mapsto [C, \OO_C(h)]
\end{tikzcd}
$$
is a birational isomorphism.

By adjunction $\overline{L}\cong \OO_C\bigl(4h-E_1-\cdots-E_{12}\bigr)$ and the restriction map $\bigl|\OO_X(C)\bigl|\rightarrow \bigr|\OO_C(C)\bigr|=\bigl|\overline{L}^2\bigr|$ is dominant, in particular every element of the linear system $|\overline{L}^2|$ is cut out by an plane octic singular at $o_1, \ldots, o_{12}$.

\vskip 4pt

Note that the resolution of the scheme $\xi=o_1+\cdots+o_{12}$ of $12$ general points in $\PP^2$ has the following form described by the Hilbert-Burch theorem:

$$
\begin{tikzcd}[column sep=18pt]
0\longrightarrow  \OO_{\PP^2}(-6)^2 \longrightarrow  \OO_{\PP^2}(-4)^3  \longrightarrow   \mathcal{I}_{\xi}\longrightarrow 0.
\end{tikzcd}
$$
Precisely, one picks general quadratic forms $a, b,c, a', b',c'\in H^0\bigl(\PP^2, \OO_{\PP^2}(2)\bigr)$ and then $\xi$ can be described by the following condition in $\PP^2=\PP^2_{[x_1, x_2, x_3]}$:
\begin{equation}\label{eq:hilbert_burch}
\mbox{rk} \left(\begin{array}{rrr}
a(x_1, x_2, x_3)  & b(x_1, x_2, x_3)  & c(x_1, x_2, x_3)\\
a'(x_1, x_2, x_3)  & b'(x_1, x_2, x_3)   & c'(x_1, x_2, x_3)
\end{array}\right)\leq 1
\end{equation}
The quadruple cover $q\colon X\rightarrow \PP^2$ defined by (\ref{eq:degree4}) is then the resolution of the rational map $\PP^2\stackrel{4:1} \dashrightarrow \PP^2$ given by $(ab'-a'b, \ ac'-a'c, \ bc'-b'c)$.

\vskip 5pt

Assume $\bigl[C, L,\eta\bigr]\in \mathcal{RG}_8^2$ is an element such that $H^0(C, L\otimes \eta)\neq 0$, that is, $[C, \eta]\in \cD_9$.  Write
$\overline{L}\otimes \eta=\OO_C(y_1+\cdots+y_8)$, therefore $\overline{L}^{\otimes 2}\cong \OO_C(2y_1+\cdots+2y_8)$. As long as $C$ is a Petri general curve, it follows that there exists a curve on $X$
$$
\begin{tikzcd}[column sep=22pt]
C'\in \bigl|8h-2E_1-\cdots-2E_{12}\bigr| \ \ \mbox{ such that  } \ \  C'\cdot C=2y_1+\cdots+ 2y_8.
\end{tikzcd}
$$

As explained in the Introduction, we  consider the pencil $R\subseteq \rr_9$ generated by the curves $C$ and $C'$. Its general element is of the form $\bigl[C_t, \OO_{C_t}\bigl(4h-E_1-\cdots-E_{12}\bigr)(-y_1-\cdots-y_8)\bigr]$.

Working globally, we can   consider the closure $\mathcal{Y}$ inside $\P \times_{\mathrm{Hilb}^{12}(\PP^2)} \P$ of the following locus
$$
\Bigl\{(\xi, C, C')\in \P \times_{\mathrm{Hilb}^{12}(\PP^2)} \P : C, C'\in \bigl|8h-2E_1-\cdots-2E_{12}\bigr|, \  \  C'\neq C \   \  \mbox{  with }\  C\cdot C'=2y_1+\cdots+2y_8\Bigr\}.
$$
The condition defining $\mathcal{Y}$ can be reformulated as saying that there exist points $y_1, \ldots, y_8\in X$, such that the curves $C$ and $C'$ intersect non-transversally at $y_1, \ldots, y_8$.
Observe that the quotient $\mathcal{Y}\dblq SL(3)$ has two $\PP^1$-bundle structures over the Brill-Noether divisor
\begin{equation}\label{eq:P1_fibration}
\begin{tikzcd}[column sep=30pt,row sep=28pt]
& \mathcal{Y}\dblq SL(3) \ar[dl, "\upsilon_1"] \ar[dr, "\upsilon_2"] &
\\
\dd_9  &  & \dd_9
\end{tikzcd}
\end{equation}
defined by forgetting either  $C$ or $C'$, that is, $\upsilon_1\bigl(\xi, C, C'\bigr):=\bigl[C, \omega_C(-h)(-y_1-\cdots-y_8)\bigr]$ respectively $\upsilon_2\bigl(\xi, C, C'):=\bigl[C', \omega_{C'}(-h)(-y_1-\cdots-y_8)\bigr]$. The rational curve $R$ passing through a general point $[C, \eta]$ in a component of $\dd_9$ corresponds precisely to  $\upsilon_2\bigl(\upsilon^{-1}_1\bigl([C, \eta]\bigr)\bigr)$.

We make now the assumption $(\dag)$ that for a given element $[C, L, \eta]$ as above,  every curve in the pencil $R$ spanned by $C$ and $C'$ is irreducible and at most $1$-nodal. Subject to this assumption, which we establish later along at least one irreducible component of the divisor $\dd_9$, we can compute the intersection numbers of $R$ and establish (\ref{eq:intnumbers2}) from the Introduction.

\begin{theorem}\label{thm:int_numbers3}
Assuming $(\dag)$ holds,  one has the following intersection numbers:
$$
\begin{tikzcd}[column sep=32pt]
R\cdot \lambda=9, \ \ R\cdot \delta_0'=47, \ \  R\cdot \delta_0^{\mathrm{ram}}=8, \ \  R\cdot \delta_0''=0.
\end{tikzcd}
$$
\end{theorem}
\begin{proof}
The assumption $(\dag)$ implies that every curve in the pencil $R=\{C_t\}_{t\in \PP^1}$ spanned by $C$ and $C'$ is irreducible and nodal. Furthermore, all curves in $R$ have a common tangent line $\ell_i$ at each of the points $y_i$ for $i=1, \ldots, 8$. We consider the blow-up of the rational surface $X$ at the points $y_1, \ldots, y_8$, as well as the points $\ell_1, \ldots, \ell_8$ regarded as points on the exceptional divisors $E_{y_1}$, \ldots, $E_{y_8}$ and set $X':=\mathrm{Bl}_{\{y_1, \ldots, y_8, \ell_1, \ldots, \ell_8\}}(X)$. We have an induced fibration
$$u\colon X'\longrightarrow \PP^1,$$
where $u^{-1}(t)$ is the proper transform of $C_t$. Note that in the pencil $R$ there exists for each $i=1, \ldots, 8$ precisely one curve $C_i=C_{t_i}$ that is singular at $y_i$ (and smooth at the remaining points $y_j$, with $j\neq i$). In this case $u^{-1}(t_i)=C'_{t_i}+E_{y_i}$, where $C'_{t_i}$ is a smooth curve of genus $8$ meeting $E_{y_j}$ transversally at two distinct points. The Prym curve structure on $u^{-1}(t_i)$ is provided by a line bundle $\eta_{t_i}\in \mbox{Pic}^0\bigl(u^{-1}(t_i)\bigr)$ such that $\eta_{E_{y_i}}\cong \OO_{E_{y_i}}(1)$, that is, each point $u^{-1}(t_i)$ gives rise to a point in the boundary divisor $\Delta_0^{\mathrm{ram}}$. It is also clear that these are the only points in the pencil $R$, where the sheaf inducing the Prym structure on the plane curve $\epsilon\bigl(u^{-1}(t)\bigr)$ is not locally free and that the intersection of $R$ with $\Delta_0^{\mathrm{ram}}$ at each of the points $[u^{-1}(t_i), \eta_{u^{-1}(t_i)}]$ is transversal, therefore
\begin{equation}\label{eq:ram_nr}
R\cdot \delta_0^{\mathrm{ram}}=8.
\end{equation}

To evaluate the remaining intersection numbers is now relatively easy. Observe first that $\chi\bigl(X', \OO_{X'}\bigr)=\chi(X, \OO_X)=1$, therefore
\begin{equation}\label{eq:lambda_nr}
R\cdot \lambda=\chi\bigl(X', \OO_{X'}\bigr)+g-1=9.
\end{equation}
Furthermore, $K_{X'}^2=K_X^2-16=K_{\PP^2}-(12+16)=-19$. Then applying Noether's formula we write $c_2(X')=12\chi\bigl(X', \OO_{X'}\bigr)-K_{X'}^2=12+19=31$, implying that
\begin{equation}\label{eq:bdry_nr}
R\cdot \bigl(\delta_0'+\delta_0''+2\delta_0^{\mathrm{ram}}\bigr)=c_2(X')+4(g-1)=31+4\cdot 8=63.
\end{equation}
Clearly $R\cdot \delta_0''=0$, for a point in the intersection would imply the existence of a plane quartic passing through all the points $o_1, \ldots, o_{12}$ and $y_1, \ldots, y_8$, which yields $\eta\cong \OO_C$, which is a contradiction. Putting this, (\ref{eq:ram_nr}) and (\ref{eq:bdry_nr}) together, we obtain $R\cdot \delta_0'=63-2\cdot 8=47$, which finishes the proof.
\end{proof}

\begin{corollary}\label{cor:inters}
One has $R\cdot K_{\rr_9}=-1$ and $R\cdot [\dd_9]=102$.
\end{corollary}
\begin{proof}
This is a direct consequence of Theorems \ref{thm:int_numbers3}, \ref{thm:classdege} and of (\ref{eq:canonical1}).
\end{proof}

Doe to the (unlikely) possibility that the Brill-Noether divisor $\dd_9$ may be reducible, we need a somewhat stronger statement than Corollary \ref{cor:inters} to conclude that $K_{\rr_9}$ is not pseudo-effective.

\begin{theorem}\label{thm:anycomponent}
Let $\ddd$ be any irreducible component of the  divisor $\dd_9$. Then $R\cdot \ddd\geq 0$.
\end{theorem}
\begin{proof} We consider an irreducible component $\ddd$ of $\dd_9$ and write
$$[\ddd]=a\lambda-b_0'\delta_0'-b_0''\delta_0''-b_0^{\mathrm{ram}}\delta_0^{\mathrm{ram}}-\sum_{i=1}^4\bigl(b_i\delta_i+b_{9-i}\delta_{9-i}+b_{i:9-i}
\delta_{i:9-i}\bigr) \in \mbox{Pic}(\rr_9).$$
We may clearly assume that $\ddd\neq \Delta_0''$. We apply Theorem \ref{thm:nikulin} and observe that $\ddd$ is disjoint from a general pencil $\Xi_9$ of Prym curves on a general Nikulin surface. Using (\ref{eq:Nikulin_pencil}) we obtain that
\begin{equation}\label{eq:nikulin4}
10a-56b_0'-8b_0^{\mathrm{ram}}=0.
\end{equation}
Note that $b_0''\geq 0$. Indeed, one considers the following sweeping curve $A_0''$ of the divisor $\Delta_0''$. Fix a general curve $[C,p]\in \cM_{8,1}$ and consider the family of Prym curves
$$A_0'':=\Bigl\{\bigl[C_{xp}= C/x\sim p, \ \eta_{xp}\bigr]: x\in C, \ \eta_{xp}\in \mbox{Pic}^0(C_{xp})[2], \ \  \nu^*(\eta_{xp})\cong \OO_C\Bigr\}.$$
Here $\nu\colon C\rightarrow C_{xp}$ denotes the normalization. Then $A_0''\cdot \ddd=(2g-2)b_0''-b_1\geq 0$. Since clearly $b_1\geq 0$ (use that the intersection of $\ddd$ with the test curve given by a ruling of the boundary divisor $\Delta_1\subseteq \rr_g$ is non-negative), one concludes that $b_0''\geq 0$, as claimed.

We next use the lift to $\rr_9$ of a general pencil of curves of genus $9$ on a fixed $K3$ surface  under the map $\pi\colon \rr_9\rightarrow \mm_9$. Given a general $K3$ surface $S$ with $\mbox{Pic}(S)=\mathbb Z\cdot [C]$  where $C^2=16$, we take a Lefschetz pencil
$\bigl\{C_t\bigr\}_{t\in \PP^1}$ in the linear system $\bigl|\OO_S(C)\bigr|$, then consider the curve
$$\begin{tikzcd}[column sep=22pt]
A:=\Bigl\{[C_t, \eta_t]: \eta_t\in \overline{\mathrm{Pic}}^0(C_t)[2], \ \ t\in \PP^1\Bigr\}\subseteq \rr_9.
\end{tikzcd}
$$
We record the intersection numbers of $A$ with the generators of $\mbox{Pic}(\rr_9)$, cf. \cite[Lemma 1.8]{FL}:

$$A\cdot\lambda=(g+1)(2^{2g}-1), \ \  A\cdot \delta_0'=(6g+18)(2^{2g-1}-2), \  A\cdot \delta_0''=6g+18, \  A\cdot \delta_0^{\mathrm{ram}}=(6g+18)2^{2g-2},$$
where $g=9$. The intersection numbers of $A$ with the remaining generators of $\mbox{Pic}(\rr_9)$ are all zero. Clearly
$A$ is a sweeping curve in $\rr_9$, therefore it intersects every effective divisor in $\rr_9$ non-negatively. In particular, combining the relation (\ref{eq:nikulin4}) with the inequality $A\cdot \ddd\geq 0$, we obtain that $5a\leq 36b_0'$. But then using (\ref{thm:int_numbers3}) coupled again with (\ref{eq:nikulin4}), we write
$$R\cdot \ddd=9a-47b_0'-8b_0^{\mathrm{ram}}=9a-47b_0'+56b_0'-10a=9b_0'-a\geq \Bigl(\frac{45}{36}-1\Bigr) a \geq 0,$$
since clearly $a\geq 0$ (the class $\lambda$ is big and nef), which brings the proof to an end.
\end{proof}

\subsection{The existence of a good sweeping rational curve inside $\dd_9$}\label{subsec:trans}  We are left with establishing the assumption (\dag), that plays a crucial role in the proof of both Theorems \ref{thm:int_numbers3} and \ref{thm:anycomponent}. Recall that $Z$ was defined as the degeneracy locus
inside $\mathcal{RG}^2_8$ of the morphism $\chi$ considered in (\ref{eq:sheaf_morphism}). Although not strictly needed in the proof, note that it follows from \cite[Proposition 3.4]{Br} and \cite[Theorem 2.3]{FL} that $\mathcal{RG}^2_8$ is irreducible.

\vskip 4pt

Keeping the notation from the beginning of Section \ref{sec:uniruledd}, we are going to exhibit a point $[C,L,\eta]\in Z$, inducing a plane octic
$\varphi_L\colon C\rightarrow \Gamma$ such that:

\begin{enumerate}
\item $\Gamma$ is nodal at $12$ points $o_1, \ldots, o_{12}$ and has no further singularities.
\item If $\overline{L}\otimes \eta\cong \OO_C(y_1+\cdots+y_8)$, the points $y_1, \ldots, y_8$ are pairwise distinct and disjoint from the set
$\bigl\{\varphi_L^{-1}(o_1), \ldots, \varphi_L^{-1}(o_{12})\bigr\}$.
\item Each curve in the pencil spanned by $\Gamma$ and $\Gamma'$ is irreducible and nodal.
\end{enumerate}

Since each of the conditions (1)-(3) is open in $Z\subseteq \mathcal{RG}^2_8$ and each component of $Z$ maps generically finite under the map $\sigma$ onto a component of $\dd_9$, exhibiting \emph{one} such point $[C, L, \eta]$, implies the existence of a component $\ddd$ of $\dd_9$, for which Theorem \ref{thm:anycomponent} can be applied.

We begin with the following observation:
\begin{lemma}\label{lemma:red}
Fix a general point $\xi \in \mathrm{Hilb}^{12}(\PP^2)$. Then the locus of reducible curves in the linear system $\bigl|8h-2E_1-\cdots-2E_{12}\bigr|$ on $X$ has codimension at least $4$.
\end{lemma}
\begin{proof}
Assume $C=C_1+C_2$ is such a reducible curve. Since the points $o_1, \ldots, o_{12}\in \PP^2$ are general, there is no cubic passing through all of them, nor a plane septic curve nodal at each of these points.  Therefore the possibility yielding the largest number of moduli is when  both $C_1$ and $C_2$ are elements of the linear system $\bigl|4h-E_1-\ldots-E_{12}|$. But such curves depend on at most $2\cdot \mbox{dim } \bigl|4h-E_1-\cdots-E_{12}\bigr|=4$ parameters.
\end{proof}

\begin{lemma}\label{lemma:no_triple}
There exists a point $[C, L, \eta]\in Z$, such that no curve in the pencil spanned by $C$ and $C'$ has either cusps or singularities of order at least $3$.
\end{lemma}

\begin{proof} We consider a general point $\xi=o_1+ \cdots+ o_{12}\in \mbox{Hilb}^{12}(\PP^2)$ such that the quadruple cover $q\colon X\rightarrow \PP^1$ considered in (\ref{eq:degree4}) has a branch curve with only nodes and cusps as singularities. We choose a general element
$$D_0\in \bigl|4h-E_1-\cdots-E_{12}\bigr|,$$
that is, $D_0$ is a smooth quartic curve passing through $o_1, \ldots, o_{12}$. Further, we take a general pencil of quartics $\Lambda:=\bigl\{D_t\}_{t\in \PP^1}$ in the linear system $\bigl|4h-E_1-\cdots-E_{12}\bigr|$, then consider the pencil of \emph{reducible} octics in $X$
$$\bigl\{C_t=D_0+D_t\subseteq X: t\in \PP^1\bigr\}$$
having $D_0$ as a base component. We may assume that each curve $D_t$ is irreducible and at most $1$-nodal and that the intersection of $D_0$ with two generators of the pencil $\Lambda$ is everywhere transverse. It thus follows that no curve $D_0+D_t$ in $X$ has cusps or singularities of order at least $3$. Observe however, that in this pencil there are $12$ curves with a tacnode, corresponding to the situation when the plane quartics $\epsilon(D_0)$ and $\epsilon(D_{t_i})$ have a common tangent at the points $o_i$, for $i=1, \ldots, 12$. Note furthermore that we can pick \emph{general} points $y_1, \ldots, y_8\in D_0$ and then with the notation of (\ref{eq:P1_fibration}), observe that the curves $D_0+D_{t_1}$ and $D_0+D_{t_2}$ obviously have non-transverse intersection at each of the points $y_1, \ldots, y_8$ (in fact they even have a common component), that is, the point $\bigl(\xi, D_0+D_{t_1}, D_0+D_{t_2}\bigr)$ can be regarded as an element of $\mathcal{Y}$.
\end{proof}

\vskip 4pt

\begin{lemma}\label{lemma:no_tacnode}
There exists a point $[C, L, \eta]\in Z$, such that no curve in the pencil spanned by $C$ and $C'$ has tacnodal singularities.
\end{lemma}
\begin{proof} We keep the notation of Lemma \ref{lemma:no_triple} and assume that we chosen $o_1+\cdots+o_{12}$ such that the quadruple cover $q\colon X\rightarrow \PP^2$ has a branch curve $\Delta\subseteq \PP^2$ with only nodes and cusps as singularities. We consider two points $\xi_a, \xi_b\in \mbox{Hilb}^2(\PP^2)$ corresponding to distinct points $a$ and $b$ and tangent directions $\xi_a\in \PP\bigl(T_a(\PP^2)\bigr)$ at $a$ and $\xi_b\in \PP\bigl(T_b(\PP^2)\bigr)$ at $b$ respectively. We denote by $\Lambda:=\bigl|\mathcal{I}_{\xi_a+\xi_b}(2)\bigr|$ the pencil of conics passing through the cluster $\xi_a+\xi_b$ in $\PP^2$ and the family of octics
\begin{equation}\label{eq:no_tacnode}
\bigl\{q^*(D): D\in \Lambda\bigr\}.
\end{equation}
Observe that for distinct elements $D,D'\in \Lambda$, the curves $q^*(D)$ and $q^*(D')$ are mutually tangent at the $8$ points in $q^*(\xi_a)$ and $q^*(\xi_b)$, therefore $\bigl(o_1+\cdots+o_{12}, q^*(D), q^*(D')\bigr)\in \mathcal{Y}$. Furthermore, no conic $D$ in the pencil $\Lambda$ is tacnodal, implying also that no curve in (\ref{eq:no_tacnode}) has a tacnode. This last claim can also be easily verified by picking $\xi\in \mbox{Hilb}^{12}(\PP^2)$ generically via (\ref{eq:hilbert_burch}). This finishes the proof.
\end{proof}

\vskip 6pt

\noindent{\emph{Proof of Theorem \ref{thm:main}.}} Applying Lemmas \ref{lemma:red}, \ref{lemma:no_triple} and \ref{lemma:no_tacnode}, we conclude that there exists an irreducible component $\ddd$ of the divisor $\dd_9$ such that the assumption $(\dag)$ is satisfied for the sweeping curve $R\subseteq \ddd$, in particular Theorems \ref{thm:int_numbers3} and \ref{thm:anycomponent} can be applied. The conclusion follows as described in the Introduction. The canonical class $K_{\rr_9}$ is not pseudo-effective. Indeed, otherwise we write $K_{\rr_9}\equiv a\cdot \ddd+M$, where $a\geq 0$ and $M$ is a pseudo-effective $\mathbb R$-divisor class on $\rr_9$ not containing $\ddd$ in its support, therefore $R\cdot M\geq 0$. It follows that $0>R\cdot K_{\rr_9}=aR\cdot \ddd+ R \cdot M\geq 0$, a contradiction. Applying \cite{BDPP}, it follows that $\rr_9$ is uniruled.

\hfill $\Box$

\bibliographystyle{amsplain}

\begin{thebibliography}{10}


\bibitem{ACGH} E. Arbarello, M. Cornalba, P. Griffiths and J. Harris, {\em{Geometry
of algebraic curves}}, Grundlehren der mathematischen Wissenschaften
267, Springer.

\bibitem{BCF} E. Ballico, C. Casagrande and C. Fontanari,
{\em{Moduli of Prym curves}}, Documenta Mathematica \textbf{9}
(2004), 265--281.

\bibitem{B} A. Beauville, {\em{Prym varieties and the Schottky
problem}}, Inventiones Math. \textbf{41} (1977), 149--96.

\bibitem{BDPP} S. Boucksom, J.P. Demailly, M. P\u{a}un and T. Peternell, {\em{The pseudo-effective cone of a compact K\"ahler manifold and varieties of negative Kodaira dimension}},  Journal of Algebraic Geometry \textbf{22} (2013), 201--248.

\bibitem{Br} G. Bruns, {\em{$\rr_{15}$ is of general type}}, Algebra \& Number Theory \textbf{10} (2016), 1949--1964.

\bibitem{CMGHL} S. Casalaina Martin, S. Grushevsky, K. Hulek and R. Laza, 	{\em{Extending the Prym map to toroidal compactifications of the moduli space of abelian varieties}}
    Journal of the European Math. Society \textbf{19} (2017), 659--723.

\bibitem{Ca} F. Catanese, {\em{On the rationality of certain moduli spaces
related to curves of genus $4$}}, Springer Lecture Notes in Mathematics \textbf{1008} (1983), 30-50.

\bibitem{CEFS} A. Chiodo, D. Eisenbud, G. Farkas and F.-O. Schreyer, {\em{Syzygies of torsion bundles and the geometry of the level $\ell$ modular varieties over $\overline{\mathcal{M}}_g$}}, Inventiones Math. \textbf{194} (2013), 73--118.

\bibitem{Cor} M. Cornalba, {\em{Moduli of curves and theta-characteristics}},
in: Lectures on Riemann surfaces (Trieste, 1987), 560--589.

\bibitem{Dol} I. Dolgachev, {\em{Rationality of $\cR_2$ and $\cR_3$}}, Pure and Applied Math. Quarterly \textbf{4} (2008), 501--508.


\bibitem{Don} R. Donagi, {\em{The unirationality of $\mathcal{A}_5$}},
Annals of Math. \textbf{119} (1984), 269--307.


\bibitem[EH2]{EH2} D. Eisenbud and J. Harros, {\em{Divisors on general curves and cuspidal rational curves}}, Inventiones Math. \textbf{74} (1983), 371--418.

\bibitem{EH} D. Eisenbud and J. Harris, \ {\em{The Kodaira dimension  of the
moduli space of curves of genus $23$}}, Inventiones Math.
\textbf{90} (1987), 359--387.


\bibitem{F1} G. Farkas, {\em{Koszul divisors on moduli spaces of curves}}, American Journal of Math.
\textbf{131} (2009), 819--869.


\bibitem{FJP1} G. Farkas, D. Jensen and S. Payne, {\em{The Kodaira dimension of $\mm_{22}$ and $\mm_{23}$}}, arXiv:2005.00622.

\bibitem{FJP2} G. Farkas, D. Jensen and S. Payne, {\em{The non-abelian Brill-Noether divisor on $\mm_{13}$ and the Kodaira dimension of $\rr_{13}$}}, Geometry \& Topology \textbf{28} (2024), 803--866.

\bibitem{FK} G. Farkas and M. Kemeny, {\em{The generic Green-Lazarsfeld Secant Conjecture}}, Inventiones Math. \textbf{203} (2016), 265--301.

\bibitem{FL} G. Farkas and K. Ludwig, {\em{The Kodaira dimension of the moduli
space of Prym varieties}}, Journal of the European Mathematical Society \textbf{12} (2010), 755--795.

\bibitem{FV1} G. Farkas and A. Verra, {\em{Moduli of theta-characteristics via Nikulin surfaces}}, Mathematische Annalen \textbf{354} (2012), 465--496.

\bibitem{FV4} G. Farkas and A. Verra, {\em{The classification of universal Jacobians over the moduli space of curves}}, Commentarii Math. Helv. \textbf{88} (2013), 587--611.

\bibitem{FV3} G. Farkas and A. Verra, {\em{The geometry of the moduli space of odd spin curves}},  Annals of Math. \textbf{180} (2014), 927--970.

\bibitem{FV2} G. Farkas and A. Verra, {\em{Prym varieties and moduli of polarized Nikulin surfaces}},
Advances in Mathematics \textbf{290} (2016), 314--328.

\bibitem{vGS} B. van Geemen and A. Sarti, {\em{Nikulin involutions on $K3$ surfaces}}, Math. Zeitschrift \textbf{255} (2007), 731--753.

\bibitem{GKP} D. Greb, S. Kebekus and T. Peternell, {\em{Movable Curves and Semistable Sheaves}}, International Math. Research Notices \textbf{2} (2016), 536--570.

\bibitem{HM} J. Harris and D. Mumford, {\em{On the Kodaira
dimension of $\mm_g$}}, Inventiones Math. \textbf{67} (1982), 23--88.

\bibitem{HL} D. Huybrechts and M. Lehn, {\em{The geometry of moduli spaces of sheaves}}, Second Edition, Cambridge University Press 2010.

\bibitem{ILS} E. Izadi, M. Lo Giudice and G. Sankaran, {\em{The
moduli space of \'etale double covers of genus $5$ is unirational}}, Pacific Journal of Mathematics \textbf{239} (2009), 39--52.

\bibitem{La} R. Lazarsfeld, {\em{Brill-Noether-Petri without degenerations}}, Journal of Diff. Geometry \textbf{23} (1986), 299--307.

\bibitem{Mu} S. Mukai, {\em{On the moduli space of bundles on $K3$ surfaces}}, in: Vector bundles on algebraic varieties, Tata Institute of Fundamental Research Studies in Mathematics, vol. 11 (1987), 341--413.

\bibitem{V1} A. Verra, {\emph{A short proof of the unirationality of $\cA_5$}}, Indagationes Math. \textbf{46} (1984), 339--355.

\bibitem{V2} A. Verra, {\em{On the universal principally polarized abelian variety of dimension $4$}}, in: Curves and abelian varieties (Athens, Georgia, 2007)  Contemporary Mathematics vol. 345, 2008, 253--274.

\end{thebibliography}

\end{document}